\documentclass[a4paper]{article}

\usepackage[english]{babel}
\usepackage[utf8]{inputenc}
\usepackage{amsmath}
\usepackage{amssymb}
\usepackage{amsthm}
\usepackage{amscd}
\usepackage{amsfonts}
\usepackage{newlfont}
\usepackage{enumerate}

\newtheorem{thm}{Theorem}[section]
\newtheorem{example}[thm]{Example}
\newtheorem{lem}[thm]{Lemma}
\newtheorem{prop}[thm]{Proposition}
\newtheorem{definition}[thm]{Definition}
\newtheorem{rem}[thm]{Remark}
\newtheorem{corol}[thm]{Corollary}

\title{On strong tangential transversality \thanks{This work was partially supported by
		the Sofia University "St. Kliment Ohridski" fund "Research \& Development"
		under contract 80-10-133/25.04.2018  and by the Bulgarian National Scientific Fund under Grant DNTS/Russia 01/9/23.06.2017.}}

\author{Mira Bivas\footnotemark[2]\ \footnotemark[3]\ \footnotemark[4]
	\and Mikhail Krastanov\footnotemark[2]\ \footnotemark[3]\ \footnotemark[5]
	\and Nadezhda Ribarska\footnotemark[2]\ \footnotemark[3]\ \footnotemark[6]
}

\date{}

\begin{document}
	
	\maketitle
	
	\section{Introduction}
	
	Transversality is a classical concept of mathematical analysis and differential topology. Recently, it has proven to be useful in variational analysis as well. As it is stated in \cite{Ioffe}, the transversality-oriented language is extremely natural and convenient in some parts of variational analysis, including subdifferential calculus and nonsmooth optimization.
	
	We arrived to the study of transversality of sets when investigating Pontryagin's type maximum principle for  optimal control problems with terminal
	constraints in infinite dimensional state space. In order to prove a nonseparation result (if one can not separate the approximating cones of two closed sets  at a common point  and, moreover, the cones have nontrivial intersection, then the sets can not be separated as well) we introduced the notion of uniform tangent set (cf. \cite{KR17}). It happened to be very useful for obtaining necessary conditions for  optimal control problems in infinite dimensional state space, because the diffuse variations (which are naturally defined and easy to calculate) form a uniform tangent set to the reachable set of a control system. The present manuscript and \cite{BKRtt} are  an effort to understand the relation of our study to the established results and methods of nonsmooth optimization. As we arrived to some known notions and results using our approach, the proofs of the known theorems are completely different from the classical ones and, moreover, we found some new and surprising results. Our proofs do not use variational principles and we are concentrated mainly on tangential conditions in the primal space. 
	
	Here we introduce the notion of \textit{strong tangential transversality} in Banach spaces. It is characterized in the primal space -- its formulation involves uniform tangent sets. It is a convenient sufficient condition for tangential transversality (cf. \cite{BKRtt}).  Unlike it, strong tangential transversality is appropriate for obtaining results not only using tangents, but also using normals. Since uniform tangent sets are closely connected to the Clarke subdiffrential (cf. \cite{KR17} and \cite{BRV17}), the main results concern the Clarke tangent and normal cones. To be more precise, strong tangential transversality implies an intersection property
	for the Clarke tangent and normal cones and a rather general sum rule for the Clarke
	subdifferential.
	
	For the sake of completeness, we remind the definition of tangential and normal intersection property from \cite{BKRtt}. Let $T_C(x)$ be the tangent cone (in some sense -- Bouligand, derivable, Clarke, ...) to the closed subset $C$ of the Banach space $X$ at $x\in C$ and $N_C(x)$ be the normal cone (in some sense -- proximal, limiting, $G$-normal, Clarke, ...) to the closed set $C\subset X$ at $x\in C$.
	\begin{definition}
		Let $A$ and $B$ be closed subsets of the Banach space $X$ and let $x_0$ belong to $A\cap B$.	
		We say that $A$ and $B$ have tangential intersection property at $x_0$ with respect to the type(s) of the approximating cones $T_A(x_0)$ and $T_B(x_0)$ if
		$$T_{A\cap B}(x_0) \supset T_A(x_0)\cap T_B(x_0) \, . $$
		
		We say that $A$ and $B$ have normal intersection property at $x_0$ with respect to the type(s) of the normal cones $N_A(x_0)$ and $N_B(x_0)$ if
		$$N_{A\cap B}(x_0) \subset N_A(x_0) + N_B(x_0)$$
		and the right-hand side of the above inclusion is weak$^*$-closed.
	\end{definition}
	
	In the next paragraphs we provide a short history of the development of the calculus for the Clarke subdifferential. Using our approach, we are able not only to generalize these results, but also to prove them in a unified way.
	
	In 1973, Clarke \cite{Clarke73} introduces concepts of subgradient and subdifferential for non-convex extended real-valued functions, now known as generalized gradient and Clarke subdifferential. They were initially introduced for Lipschitz functions in Banach spaces and in \cite{Clarke73} calculus for such functions is developed as well. Clarke's results also provide a direct formula for tangent and normal cones in finite-dimensional spaces, further studied and developed in Banach spaces in \cite{Rock79rn} and \cite{HU79}. 
	
	Rockafellar extends Clarke's generalized gradient and subdifferential for arbitrary functions in Banach spaces in \cite{Rock80}. In the same paper he introduces  directionally Lipschitz functions and the hypertangent cone. In \cite{Rock79dl} Rockafellar proves sum rule for the Clarke subfifferential in Banach spaces if one of the functions is directionally Lipschitz and tangent and normal intersection properties for the Clarke tangent and normal cones, if one of the sets is epi-Lipschitzian (introduced in \cite{Rock79rn}, see also \cite{Rock80}). The hypotheses in these result can be weakened, if the space is assumed to be finite-dimensional, which is done in \cite{WBor87}. In finite dimensions, Aubin and Ekeland \cite{AuEke84} also obtain a normal intersection property for the Clarke normal cone, if the difference of the corresponding Clarke tangent cones is the whole space.
	
	In 1985 Borwein and Strojwas \cite{BS85} introduce the class of compactly epi-Lipschitz sets as appropriate for investigating tangential approximations of the Clarke tangent cone in Banach spaces. Then, Borwein \cite{Borw87} introduces epi-Lipschitz-like sets and shows that for them the Clarke normal cone  adequately measures boundary behaviour (i.e. the Clarke normal cone is weak-star locally compact). In \cite{BS89} Borwein and Strojwas generalize the tangential intersection property of Rockafellar \cite{Rock80}, again if one of the sets is epi-Lipschitzian.
	
	Clarke and Raissi \cite{Clarke90} prove a tangential intersection formula for the Clarke tangent cone in Hilbert spaces and provide a sufficient condition for the corresponding normal intersection formula in a remark. In the proof considerably is used the fact that in Hilbert spaces the Clarke normal cone is the closed convex hull of the proximal  normal cone. This approach is extended to Banach spaces in \cite{Jou95}, where Ioffe's $G$-normal cone is used instead of the proximal one. Jourani and Thibault \cite{JouThi95} prove normal intersection property for the Clarke normal cones and a sum rule for the Clarke subdifferential, using the corresponding results for Ioffe's $G$-normal cone and geometric approximate subdifferential.
	
	Ciligot-Travain \cite{CiTr99} proves a intersection formula for the normal cone associated with the hypertangent cone without using directionally Lipschitz assumptions.
	
	As it is seen from above, sufficient conditions for normal intersection property for the Clarke normal cone (sum rule for the Clarke subdifferential) involve strong assumptions on one of the sets (epigraphs) - existence of a hypertangent, epi-Lipschitz-like (directionally Lipschitz, Lipschitz-like). We prove these results imposing transversality-type assumption on the sets involved. This has already been done for the limiting and $G$-normal cones (subdifferentials) -- see \cite{Penot} and \cite{IoffeBook}, but it has not been known for the Clarke normal cone (subdifferential).
	
	The paper is organized as follows: the definition of a uniform tangent set is given in the second section and some known properties of such sets are summarized. Moreover, it is proven that for each epi-Lipschitzian   set at the reference point there exists an open   uniform tangent set generating the corresponding Clarke tangent cone. The notion of strong tangential transversality of two sets is introduced in the third section and its relation to tangential transversality (cf. \cite{BKRtt}) is clarified. A tangential intersection property (for uniform tangent sets) and normal intersection properties (for Clarke normal cones) are proven. An example illuminating the limits of validity of the Clarke normal intersection property is given. In the last section a sum rule for the Clarke subdifferential is presented.
	
	Throughout the paper if $Y$ is a Banach space, we will denote by ${\mbox{\bf B}_Y}$ [$ \bar{\mbox{\bf{B}}}_Y$] its open [closed] unit ball, centered at the origin. The index could be omitted if there is no ambiguity about the space.
	
	\textbf{Acknowledgments.} We are grateful to Prof. S. Troyanski and Prof. D. Kutzarova for their helpful suggestions concerning Example \ref{CEL}.
	
	\section{Uniform tangent sets}
	
	The following definitions are from \cite{KR17}:
	
	\begin{definition} \label{UnCone}
		Let   $S$ be a  closed subset of $X$ and $x_0$ belong to $S$. We
		say that the bounded  set  $D_S(x_0)$ is a uniform  tangent  set to $S$ at
		the point $x_0$ if for each $\varepsilon >0$ there exists
		$\delta>0$ such that for each $v \in D_S(x_0)$
		and for each point $x \in S \cap (x_0 + \delta \bar {\mbox{\bf
				B}})$ one can find $\lambda>0$   for which      $S \cap (x+t (v + \varepsilon \bar {\mbox{\bf
				B}}))$ is non empty  for each $t \in
		[0,\lambda]$.
	\end{definition}
	
	\begin{definition} \label{UnConeSequences}
		Let   $S$ be a closed subset of $X$ and $x_0$ belong to $S$. We
		say that the bounded set  $D_S(x_0)$ is a sequence  uniform tangent
		set to $S$ at the point $x_0$ if for each $\varepsilon >0$ there
		exists $\delta>0$ such that for each $v \in D_S(x_0)$ and for each point $x \in S \cap (x_0 + \delta \bar
		{\mbox{\bf B}})$ one can find a sequence of positive reals
		$t_m\to 0$  for which $S
		\cap (x+t_m (v + \varepsilon \bar {\mbox{\bf B}}))$ is non empty
		for each positive integer $m$.
	\end{definition}
	
	The next theorem is the main result from \cite{BRV17}.
	
	\begin{thm}\label{equiv_def}
		Let   $S$ be a closed subset of $X$ and $x_0$ belong to $S$. The following are equivalent
		\begin{enumerate}[(i)]
			\item $D_S(x_0)$ is a uniform tangent set to $S$ at the point $x_0$
			\item $D_S(x_0)$ is a sequence  uniform tangent set to $S$ at the point $x_0$
			\item for each $\varepsilon >0$ there exist
			$\delta>0$  and $\lambda>0$ such that for each $v \in D_S(x_0)$
			and for each point $x \in S \cap (x_0 + \delta \bar {\mbox{\bf
					B}})$ the set $S \cap (x+t (v + \varepsilon \bar {\mbox{\bf
					B}}))$ is non empty  for each $t \in
			[0,\lambda]$.
		\end{enumerate}
	\end{thm}
	
	The basic properties of uniform tangent sets are gathered in the next proposition.
	
	\begin{prop}\label{uts}
		Let $S$ be a closed subset of  $X$ and let
		$x_0\in S$. Let $D_S(x_0)$ be  an  uniform tangent set to $S$ at the point $x_0 $. Then, the following hold true:
		\begin{enumerate}[(i)]
			\item if $c>0$ is fixed,  $c \, D_S(x_0)$ is a uniform  tangent set    to $S$ at $x_0$
			\item if $D'_S(x_0) \subset D_S(x_0)$,  $D'_S(x_0)$ is a  uniform  tangent set    to $S$ at $x_0$
			\item if $D'_S(x_0)$ is  another  uniform  tangent set    to $S$ at $x_0$, then   $D_S(x_0) \cup D'_S(x_0)$ is a   uniform  tangent set   to $S$ at $x_0$
			\item the closure of $D_S(x_0)$ is a  uniform  tangent set    to $S$ at $x_0$
			\item $\overline{\mbox{co}} \, D_S(x_0)$ is a  uniform  tangent set    to $S$ at $x_0$
			\item if $S$ is convex, then $(S-x_0) \cap M \bar {\mbox{\bf B}} $ is a  uniform  tangent set    to $S$ at $x_0$ for every $M>0$.
		\end{enumerate}
		
	\end{prop}
	
	\begin{proof}
		The statements $(i)-(iii)$ follow directly from the definition of a uniform tangent set. The statements $(iv)-(v)$ are from Lemma 2.5 in \cite{KR17}. The proof of $(vi)$ is contained in the proof of Corollary 2.8 in \cite{KR17}.
	\end{proof}

	We are going to use the original definition of the Clarke tangent cone in Banach spaces from \cite{Rock80}:
	\begin{definition} \label{ClarkeCone}
		Let   $S$ be a  closed subset of $X$ and $x_0$ belong to $S$. We
		say that  $v\in X$  is a tangent vector to $S$ at
		the point $x_0$ if for each $\varepsilon >0$ there exist
		$\delta>0$ and $\lambda>0$ such that for each point $x \in S \cap (x_0 + \delta \bar {\mbox{\bf
				B}})$  the set     $S \cap (x+t (v + \varepsilon \bar {\mbox{\bf
				B}}))$ is non empty  for each $t \in
		[0,\lambda]$.
		The set of all tangent vectors to $S$ at $x_0$ is called Clarke tangent cone and is denoted by $\hat T_S(x_0)$.
	\end{definition}
	
	Theorem \ref{equiv_def} implies
	\begin{corol}\label{corol_Clarke}
		Let $S$ be a closed subset of $X$ and let
		$x_0\in S$. Let $D_S(x_0)$ be  an  uniform tangent set to $S$ at the point $x_0 $. Then, $D_S(x_0)$ is a subset of $\hat T_S(x_0)$, where $\hat T_S(x_0)$ is the Clarke tangent cone to $S$ at $x_0$.
	\end{corol}
	
	We obtain stronger connections between uniform tangent sets and the Clarke tangent cone. To this end, we will need the definitions below:
	
	\begin{definition}
		The conical hull of a set $S$ is defined as
		$$\mbox{cone } S := \mathbb{R}^{+}S = \{ \alpha s \ | \ \alpha > 0 \mbox{ and } s \in S \} \ .$$
	\end{definition}
	
	\begin{definition}
		A set $S$ is said to generate the cone $C$, if $C$ is the closure of the conical hull of $S$. It is denoted by $\overline{\mbox{cone}} \, S \ .$
	\end{definition}
	
	\begin{definition}[Rockafellar (1979), \cite{Rock79dl}]
		Let $S$ be a closed subset of the Banach space $X$ and $x \in S$. A vector $v\in X$ is said to be a hypertangent to $S$ at $x$ if for some $\eta >0$
		$$y+tw \in S \mbox{    for all } y\in(x+\eta  \bar{\mbox{\bf B}})\cap S, \ w\in v + \eta  \bar{\mbox{\bf B}}, \ t \in (0, \eta) \, .$$
		If $S$ has a hypertangent at $x$, then $S$ is said to be epi-Lipschitz at $x$.
	\end{definition}

	Existence of a uniform tangent set generating the Clarke tangent cone in some cases is shown below.
	
	\begin{lem}
		Let $S$ be a closed subset of the Banach space $X$ and $x \in S$. If there exists a
		hypertangent $v$ to $S$ at the point $x$, then there exists an open uniform tangent set
		$D$ to $S$  at $x$ which generates the Clarke tangent cone $\hat T_S (x)$ to $S$ at $x$.
	\end{lem}
	
	\begin{proof} Let us remind that the existence of a hypertangent to $S$ at $x$ implies that every
		element of the interior   $ int\  \hat T_S (x)$ of the Clarke tangent cone $\hat T_S (x)$ to $S$ at $x$
		is a hypertangent as well (cf. Corollary 2 in \cite{Rock80}).
		Hence for every element $v \in  {int}\ \hat T_S (x)$  with $\|v\|=1$ we set
		$$\varepsilon(v):=\sup \left\{\eta\in  (0, 1):
		\begin{array}{l}
		\  y+  \lambda (v + \eta   \mbox{\bf B})  \subset S   \mbox{
			for each } \lambda \in (0,\eta )
		\\
		\mbox{ and  for each } y \in (x+    \eta  \mbox{\bf B} )   \cap S
		\end{array}\right\}.$$  Clearly,  then
		$y+  \lambda (v + \varepsilon(v)  \mbox{\bf B})  \subset S $
		for each $ \lambda \in (0,\varepsilon(v) )$
		and  for each $ y \in (x+    \varepsilon(v)  \mbox{\bf B} )   \cap S$.
		
		We set
		$$
		D := \left\{ \varepsilon (v)  v: \ v \in int\  \hat T_S (x) \mbox{ with } \|v\|=1 \right\}.
		$$
		
		Let us fix an arbitrary $\varepsilon >0$. We set $\delta:=\varepsilon >0 $.
		Let $w$ be an arbitrary element of $D$. Then $w = \displaystyle \varepsilon(v)v  $ for some hypertangent $v$ to $S$ at $x$ with
		$\|v\|=1$, and hence $\|w\|=\varepsilon (v)$.
		
		If $   \varepsilon  \ge \varepsilon (v) = \|w\|$, then   $\mathbf{0} \in w + \varepsilon \bar{\mbox{\bf B}}$.
		Thus for each $\lambda \ge 0$ we have $y \in (y + \lambda (w + \varepsilon \bar{\mbox{\bf B}})) \cap S$.

		If $  \varepsilon  < \varepsilon (v)$, then    $  (x+\delta \bar{\mbox{\bf B}})\cap S \subseteq (x+\varepsilon (v) \bar{\mbox{\bf B}})\cap S $, and hence
		for each $y \in (x+\delta \bar{\mbox{\bf B}})\cap S$ we have that
		$
		y+  \lambda (v + \varepsilon(v) \bar{\mbox{\bf B}})) \subset S
		$ for each $\lambda \in [0, \varepsilon (v)]$.  Therefore
		for each $\lambda \in [0, \varepsilon (v)]$ we have that
		$$
		S\ni y + \lambda v = y + \frac  \lambda {\varepsilon (v)} w \subset
		y + \frac  \lambda {\varepsilon (v)}  (w + \varepsilon \bar{\mbox{\bf B}}).
		$$
		So we have proved that $D$ is a uniform tangent set to $S$ at $x$.
		Next we check that $\Sigma \cap \left({int}\ \hat T_S (x)\right) \ni v\mapsto \varepsilon (v)$ is locally Lipschitz
		continuous function with Lipschitz constant one. Let $v$ and $w$ be arbitrary hypertangents
		of norm one with $\|v-w\|< \varepsilon (v)$. We shall prove that $\varepsilon (w) \ge \varepsilon (v) - \|v-w\|$.
		Indeed, let $y \in x + (\varepsilon (v) - \|v-w\|)\mbox{\bf B} \subset x + \varepsilon (v) \mbox{\bf B}$  and let $\lambda \in (0, \varepsilon (v)- \|v-w\|)
		\subset (0 , \varepsilon (v)) $. Then
		$$
		y +\lambda (w+(\varepsilon (v) - \|v-w\|)\mbox{\bf B}) = y +\lambda (v+ (w-v)+ (\varepsilon (v) - \|v-w\|)\mbox{\bf B})\subset $$ $$
		\subset
		y +\lambda (v + \|v-w\|\mbox{\bf B}+(\varepsilon (v) - \|v-w\|)\mbox{\bf B}) = y +\lambda (v +   \varepsilon (v)  \mbox{\bf B}) \subset S.
		$$
		Therefore \begin{equation}\label{Inequality} \varepsilon (w)\ge \varepsilon (v) - \|v-w\| \mbox{ whenever }
		\|v-w\|< \varepsilon (v)
		. \end{equation}
		
		Now let us fix a hypertangent $v$ of norm one.  We prove that
		the function  $\varepsilon (\cdot)$ is Lipschitz continuous with Lipschitz constant one on the ball
		$v+ \displaystyle \frac {\varepsilon (v)} 4  \mbox{\bf B}$, intersected with $\Sigma \cap \left({int}\ \hat T_S (x)\right)$. Let  us fix two arbitrary points $w_1$ and $w_2$
		in $\left(v+ \displaystyle \frac {\varepsilon (v)} 4  \mbox{\bf B}\right)\cap \left({int}\ \hat T_S (x)\right)$. Then $\|w_i -v \| < \displaystyle \frac {\varepsilon (v)} 4 < {\varepsilon (v)} $,
		and hence $\varepsilon (w_i)\ge \varepsilon (v) - \|v-w_i\|\ge \displaystyle \frac {3\varepsilon (v)} 4$, $i=1,2$.
		Thus $\|w_2-w_1\| < \displaystyle \frac {\varepsilon (v)}2 < \displaystyle \frac {3\varepsilon (v)}4 < \displaystyle   {\varepsilon (w_1)}$,
		$i=1,2$. Applying again the  inequality (\ref{Inequality}), we obtain
		$$
		\varepsilon (w_1)\ge \varepsilon (w_2) - \|w_1 -w_2\| \mbox{ and } \varepsilon (w_2)\ge \varepsilon (w_1) - \|w_1 -w_2\|.
		$$
		Therefore $| \varepsilon (w_1)- \varepsilon (w_2)|\le \|w_1 -w_2\|$, i.e. $\varepsilon (\cdot)$ is locally Lipschitz
		continuous with Lipschitz constant one.
		
		We set
		$$
		\tilde D := \left\{ \lambda   v: \ \ v \in int\  \hat T_S (x) \mbox{ with } \|v\|=1, \mbox{ and }
		\lambda \in (0, \varepsilon (v)) \right\}.
		$$
		\mbox{\bf B}ecause $\tilde D \subset \mbox{ conv }(\{\mathbf{0} \} \cup D)$ and $D$ is a uniform tangent set to $S$ at $x$,
		$\tilde D$ is  a uniform tangent set to $S$ at $x$ as well (cf. Lemma 2.5 of \cite{KR17}).
		Moreover $\tilde D$ is open. Indeed, let us fix an arbitrary point $w \in \tilde D$. Then
		$w=\lambda v$ with $v \in \left( int\  \hat T_S (x)\right) \cap \Sigma$ and $\lambda \in (0, \varepsilon (v))$.
		Let $\delta>0$ be such that $\varepsilon (\cdot)$ is Lipschitz continuous on $\mbox{\bf B}_\delta (v) \cap \Sigma \subset \
		int \ \hat T_S (x)$. We set $\eta \in (0, \min\{\lambda, 1\})$ to satisfy the inequalities
		$$
		\frac 2 \lambda \eta < \delta, \ \left(1 + \frac 2 \lambda\right) \eta < \varepsilon (v) -\lambda \mbox{ and } \eta < \lambda \, .
		$$
		Let $w'$ be an arbitrary element of the ball $\mbox{\bf B}_\eta (w)$.
		Then  $v':= w'/\|w'\|$ is well defined and
		$$
		\|v'-v\| = \left\|\frac {w'}{\|w'\| } - \frac {w}{\|w\| } \right\| =\left\|\frac {w'}{\|w'\| } - \frac {w'}{\|w \| } +
		\frac {w'}{\|w \| }- \frac {w}{\|w\| } \right\| \le
		$$ $$
		\le \|w'\|\frac {|\|w'\|-\|w\|| }{\|w \|\|w'\| } + \frac { \|w' - w\|  }{\|w \|  } \le \frac {2 \|w' - w\|  }{\|w \|  } <
		\frac 2 \lambda \eta < \delta.
		$$
		Then $|\varepsilon (v) - \varepsilon (v')|\le \|v-v'\|< 2\eta / \lambda$, and therefore
		$\varepsilon (v) - 2\eta / \lambda \le \varepsilon (v')$.
		
		On the other hand-side, the inequalities we have imposed on $\eta$ imply that
		$\lambda+\eta< \varepsilon (v) - 2\eta / \lambda$, and thus
		$$0< \lambda':= \|w'\| = \lambda +(\lambda' - \lambda)\le \lambda + |\|w'\|-\|w\|| \le\lambda +
		\|w' - w\| < $$ $$ <\lambda + \eta < \varepsilon (v) - \frac {2\eta}  \lambda \le \varepsilon (v').
		$$
		Therefore $w' = \lambda' v'\in \tilde D$.
		This completes the proof.
	\end{proof}
	
	\begin{lem}[Lemma 2.9 from \cite{BRV17}]\label{separ}
		Let $X$ be a Banach space and let $S$ be its subset. Let $x\in S$ and let the Clarke tangent cone $\hat T_S(x)$ to $S$ at $x$ be separable.  Then there exists a uniform tangent set $D_S(x)$ to $S$ at $x$ which generates $\hat T_S(x)$.
	\end{lem}

	\section{Strong tangential transversality and intersection properties}
	
	The definition below is central for our considerations:
	
	\begin{definition}
		Let $A$ and $B$ be closed subsets of the Banach space $X$ and let
		$x_0\in A\cap B$.
		We say that $A$ and $B$ are strongly tangentially transversal at $x_0$, if there exist $D_A(x_0)$ --  uniform tangent set to $A$ at the point $x_0 $, $D_B(x_0)$ -- uniform  tangent set   to the set $B$ at the point $x_0$ and  $\rho >0$ such that
		$$\rho \bar {\mbox{\bf B}} \subset \overline{\mbox{co}} \left( D_A(x_0) - D_B(x_0) \right) \, $$
		where
		$\bar {\mbox{\bf B}}$ is the closed unit ball of $X$.
	\end{definition}
	
	The following definition is from \cite{BKRtt}:
	
	\begin{definition}
		Let $A$ and $B$ be closed subsets of the Banach space $X$.
		We say that $A$ and $B$ are tangentially transversal at $x_0 \in A \cap B$, if there exist $M>0$, $\delta >0$ and  $\eta >0$ such that for all $x^A \in (x_0 + \delta \bar {\mbox{\bf B}})\cap A$ and $x^B \in (x_0 + \delta \bar {\mbox{\bf B}})\cap B$, there exists a sequence $\{t_m\}$, $t_m\searrow 0$, such that for every $m \in \mathbb{N}$ there exist $w^A_m \in X$ with $\|w^A_m\|\le M$ and $x^A + t_m w^A_m \in A$, and $w^B_m \in X$ with $\|w^B_m\|\le M$, $x^B + t_m w^B_m \in B$ and $\|x^A-x^B + t_m(w^A_m-w^B_m)\| \le \|x^A-x^B\| - t_m \eta \, .$
	\end{definition}
	
	We show that strong tangential transversality of two sets at a common point implies tangential transversality of the sets at the reference point in the following
	
	\begin{prop}
		Let $A$ and $B$ be closed subsets of the Banach space $X$ and let $x_0\in A\cap B$.
		If $A$ and $B$ are strongly tangentially transversal at $x_0$, then $A$ and $B$ are tangentially transversal at $x_0$.
	\end{prop}
	
	\begin{proof}
		Without loss of generality, we may
		assume that the sets $D_A(x_0)$ and $D_B(x_0)$ are closed, convex, contain the origin and $D_A(x_0) - D_B(x_0)$ is dense in $\Sigma$ -- the unit sphere. Let us fix an arbitrary $\varepsilon \in (0, \frac 1 3)$. 
		
		From the definition of a uniform tangent set, applied for $D_A(x_0)$ and $D_B(x_0)$, we obtain that for $\varepsilon >0$ there exist
		$\delta>0$ and $\lambda>0$ such that for each $v \in D_A(x_0) (v \in D_B(x_0))$, for each point $x \in A \cap (x_0 + \delta \bar {\mbox{\bf B}})$ ($x \in B \cap (x_0 + \delta \bar {\mbox{\bf B}})$) and for all $t\in [0,\lambda]$, the following is fulfilled
		\begin{align}\label{def}
		&A \cap \left(x+t \left(v + { \varepsilon} \, \bar {\mbox{\bf B}}\right)\right) \neq \emptyset \\
		&\left(B\cap \left(x+t\left(v+{ \varepsilon}  \, \bar {\mbox{\bf B}}\right)\right) \neq \emptyset \nonumber\right) \, .
		\end{align}
		
		We are going to verify the definition of $A$ and $B$ -- tangentially transversal at $x_0$ with $M := \max\{ \|v\| \ | \ v \in D_A(x_0) \cup D_B(x_0) \} + \varepsilon >0$, $\eta := 1 - 3\varepsilon >0$ and $\delta >0$ from above. Let us fix $x^A \in (x_0 + \delta \bar {\mbox{\bf B}})\cap A$ and $x^B \in (x_0 + \delta \bar {\mbox{\bf B}})\cap B$. Let $t \in [0, \min\{\lambda, \|x^A - x^B\| \}]$ be arbitrary.  We set
		$$
		v:= - \ \frac {x^A - x^ B}{\|x^A - x^B
			\|} \ .
		$$
		Clearly, $\|v\|=1$, and the density of $D_A(x_0) - D_B(x_0)$ in $\Sigma$  yields that there exist
		elements $v^A \in  D_A(x_0)$ and $ v^B \in D_B(x_0)$ such that $\| v - ( v^A -  v^B)\| <  \varepsilon$. From \eqref{def}, we can find $w^A \in X$ and
		$w^B \in X$ such that
		$$
		x^A+ t w^A\in A, \ 
		x^B+ t w^B\in B \mbox{ and }
		\|  v^A- w^A\| \le \varepsilon, \ \| v^B-
		w^B\| \le \varepsilon \, .
		$$
		Then clearly $\|  w^A\| = \|w^A - v^A + v^A\| \le \|w^A -  v^A\| + \| v^A\| \le M$ and $\|
		w^B\|\le M$.
		
		Moreover, we obtain
		\begin{align*}
		\|x^A& - x^B + t
		( w^A -  w^B) \| \\
		&= \|x^A - x^B + t v + t ( (v^A - v^B)-v) +
		t ( w^A - v^A) - t (
		w^B - v^B)\| \\
		&\le \left\|x^A - x^B -  t \frac {x^A -
			x^B}{\|x^A - x^B\|}\right\| + \varepsilon t + t \|w^A -  v^A\|+ t \|w^B - v^B
		\| \\
		&\le \|x^A - x^B\| \left|1 - t \frac
		{1}{\|x^A - x^B\|}\right| +   {3\varepsilon}t
		=  \|x^A - x^B\|  - t + {3\varepsilon} t \\
		&= \|x^A - x^B\| - t \left(1- {3\varepsilon}\right) = \|x^A - x^B\| - t \eta \, ,
		\end{align*}
		which completes the proof.
	\end{proof}
	
	Next, we obtain a tangential intersection property for uniform tangent sets.
	
	\begin{prop}
		Let $A$ and $B$ be closed subsets of the Banach space $X$ and let
		$A$ and $B$ be tangentially transversal at $x_0 \in A \cap B$.
		If $D_A(x_0)$ is a uniform tangent set to $A$ at the point $x_0 $ and $D_B(x_0)$ is a uniform  tangent set  to the set $B$ at the point $x_0$, then $D_A(x_0) \cap D_B(x_0)$ is a uniform  tangent set  to the set $A \cap B$ at the point $x_0$.
	\end{prop}
	\begin{proof}
		The proof is analogous to the proof of Proposition 3.3 in \cite{BKRtt}, only the corresponding $\bar \delta$ is the same for all $v_0\in D_A(x_0) \cap D_B(x_0)$.
	\end{proof}
	
	In infinite dimensions the natural counterpart of the notion ``transversality of two cones'' is not ``the difference of the cones is the whole space'', but the following
	
	\begin{definition}
		Let $X$ be a Banach space and $C_1$, $C_2$ be two convex closed cones in it. We say that $C_1$ are $C_2$ transversal, if there exists a positive real $\rho$ such that
		\begin{equation}\label{transv_cones}
		\rho \bar {\mbox{\bf B}} \subset \overline{(C_1 \cap \bar{\mbox{\bf B}}) - (C_2 \cap \bar{\mbox{\bf B}}) } \, .
		\end{equation}
	\end{definition}
	
	This definition is equivalent to say that $C_1$ and $C_2$ are strongly tangentially transversal at the origin (as sets). 
	
	\begin{rem}
		Strong tangential transversality of two closed sets at a common point implies both tangential transversality of the sets and transversality of the Clarke tangent cones to the sets at the reference point. For example, this allows to use Theorem 3.4 from \cite{BKRtt} in order to obtain Lagrange multipliers.
	\end{rem}
	
	
	In order to prove normal intersection properties, we need the following technical proposition and the lemma below it, which is the key step.
	
	\begin{prop}\label{polars}
		Let $A$ and $B$ be closed convex subsets of the Banach space $X$, containing the origin. Then, the following hold true
		\begin{enumerate}[(i)]
			\item $(A+B)^\circ \subset A^\circ \cap B^\circ \subset 2(A+B)^\circ $
			\item $(A\cap B)^\circ \subset \overline{A^\circ + B^\circ}^{\,w^*} \subset 2(A\cap B)^\circ  \, .$
		\end{enumerate}
	\end{prop}
	\begin{proof}
		$(i)$ We have that 
		$$ A^\circ \cap B^\circ = \{x^* \in X^* \ | \ \langle x^*, a \rangle \le 1 \mbox{ and }  \langle x^*, b \rangle \le 1 \mbox{ for all } a\in A, \, b\in B \} \, ,$$
		hence
		$$ A^\circ \cap B^\circ \subset \{x^* \in X^* \ | \ \langle x^*, a+b \rangle \le 2  \mbox{ for all } a\in A, \, b\in B \} = 2(A+B)^\circ \, .$$
		For the first inclusion -- using that
		$$ (A+B)^\circ = \{x^* \in X^* \ | \ \langle x^*, a+b \rangle \le 1  \mbox{ for all } a\in A, \, b\in B \}$$
		and that the origin belongs to both $A$ and $B$, we obtain
		$$ (A+B)^\circ \subset  \{x^* \in X^* \ | \ \langle x^*, a \rangle \le 1 \mbox{ and }  \langle x^*, b \rangle \le 1  \mbox{ for all } a\in A, \, b\in B \} = A^\circ \cap B^\circ \, .$$
		
		$(ii)$ We have that 
		\begin{align*}
		A^\circ + B^\circ &= \{x^* \in X^* \ | \ \mbox{ there exist } a^*,\, b^*\in X^* \mbox{ s.t. } x^*=a^*+b^* \mbox{ and }\\
		&\langle a^*, a \rangle \le 1 \mbox{ and }  \langle b^*, b \rangle \le 1 \mbox{ for all } a\in A, \, b\in B \} \, ,	
		\end{align*}
		and 
		$$2(A\cap B)^\circ =  \{x^* \in X^* \ | \ \langle x^*, c \rangle \le 2 \mbox{ for all } c\in A\cap B \} \, ,$$
		so the second inclusion is directly obtained, using that $2(A\cap B)^\circ$ is $w^*$-closed.
		
		For the first one, let us assume the contrary -- there exists a point $x_0^*$ in $\overline{A^\circ + B^\circ}^{\,w^*} \setminus (A\cap B)^\circ$. Then, we can separate $x_0^*$ and the $w^*$-closed convex set $(A\cap B)^\circ$. That is, there exist a nontrivial $x_0\in X$ and a nonnegative real $\alpha$, such that
		$$\langle x^*, x_0\rangle \le \alpha < \langle x^*_0, x_0\rangle$$
		for all $x^* \in (A\cap B)^\circ$. Multiplying, if necessary, $x_0$ by a positive constant, we may assume that $\alpha = 1$.
		
		Due to the closedness, convexity and containment of the origin in $A$ and $B$, we have that
		$${x_0} \in \left((A\cap B)^\circ\right)_\circ = \overline{co}\, \left(\{\mathbf{0}\}\cup (A\cap B) )\right) = A\cap B \, ,$$
		which is a contradiction.
	\end{proof}
	
	\begin{lem}\label{weak_star}
		Let $C_A$ and $C_B$ be closed convex cones in the Banach space $X$. Let $C_A$ and $C_B$ be transversal.
		Then, ${C_A}^\circ+ {C_B}^\circ$ is a weak$^*$-closed subset of $X^*$.
	\end{lem}
	
	\begin{rem}
		This is a well-known result, if $C_A$ and $C_B$ are subspaces. For closed convex cones, similar assertion can be found in \cite{Jam72} with the difference that instead of condition \eqref{transv_cones}, the assumed sufficient condition is that
		$$ \rho \bar {\mbox{\bf B}} \subset \overline{(C_A \cap \bar{\mbox{\bf B}}) + (C_B \cap \bar{\mbox{\bf B}}) }$$
		for a positive real $\rho$.
	\end{rem}
	
	\begin{proof}
		Let us examine $x^* \in \overline{{C_A}^\circ+ {C_B}^\circ}^{\,w^*} \cap n \bar {\mbox{\bf B}}_{X^*}$ for some $n\in \mathbb{N}$. Then, there exist nets $\{u_\alpha^* \}_{\alpha} \subset {C_A}^\circ$ and $\{v_\alpha^* \}_{\alpha} \subset {C_B}^\circ$ such that $u_\alpha^* + v_\alpha^* \xrightarrow{w^*} x^*$ in $X^*$ and $\{u_\alpha^* + v_\alpha^* \}_\alpha \subset n \bar{\mbox{\bf B}}_{X^*}$.
		
		Since ${C_A}^\circ$ and ${C_B}^\circ$ are cones, we have that $\frac 1 n y^* \in {C_A}^\circ+ {C_B}^\circ$ if and only if $y^* \in n({C_A}^\circ+ {C_B}^\circ)$. Hence, without loss of generality we may assume that 
		$$\{u_\alpha^* + v_\alpha^* \}_\alpha \subset \bar {\mbox{\bf B}}_{X^*} \, .$$
		
		Let us denote
		$$C:= \overline{({C_A} \cap \bar {\mbox{\bf B}}) - ({C_B} \cap \bar {\mbox{\bf B}})} = \overline{({C_A} \cap \bar {\mbox{\bf B}}) + (-{C_B} \cap \bar {\mbox{\bf B}})} \, .$$
		The assumption $C \supset \rho \bar {\mbox{\bf B}} $ implies 
		\begin{equation}\label{ccirc}
		C^\circ \subset (\rho \bar {\mbox{\bf B}})^\circ = \frac 1 \rho \bar {\mbox{\bf B}}_{X^*} \, .
		\end{equation}
		
		By subsequent use of Proposition \ref{polars} $(i)$ and $(ii)$, we obtain that
		\begin{align*}
		C^\circ& = \left(\overline{({C_A} \cap \bar {\mbox{\bf B}}) + (-{C_B} \cap \bar {\mbox{\bf B}})}\right)^\circ \supset \frac 1 2 \left( ({C_A} \cap \bar {\mbox{\bf B}})^\circ \cap (-{C_B} \cap \bar {\mbox{\bf B}})^\circ \right) \\
		&\supset  \frac 1 2 \left[ \frac 1 2 \overline{({C_A}^\circ + \bar {\mbox{\bf B}}_{X^*})}^{w^*}  \cap  \frac 1 2 \overline{(-{C_B}^\circ + \bar {\mbox{\bf B}}_{X^*})}^{w^*} \right]\\
		&= \frac 1 4 \left[  \overline{({C_A}^\circ + \bar {\mbox{\bf B}}_{X^*})}^{w^*}  \cap   \overline{(-{C_B}^\circ + \bar {\mbox{\bf B}}_{X^*})}^{w^*} \right]
		\end{align*}
		Since
		$$u_\alpha^* = (u_\alpha^* + v_\alpha^*) - v_\alpha^* \in \bar {\mbox{\bf B}}_{X^*} - {C_B}^\circ \subset \overline{(-{C_B}^\circ + \bar {\mbox{\bf B}}_{X^*})}^{w^*}$$
		and
		$$u_\alpha^* \in {C_A}^\circ \subset ({C_A}^\circ + \bar {\mbox{\bf B}}_{X^*}) \subset  \overline{({C_A}^\circ + \bar {\mbox{\bf B}}_{X^*})}^{w^*} \, , $$
		it follows that
		$$\{u_\alpha^*\}_\alpha \subset  \overline{({C_A}^\circ + \bar {\mbox{\bf B}}_{X^*})}^{w^*}  \cap   \overline{(-{C_B}^\circ + \bar {\mbox{\bf B}}_{X^*})}^{w^*} \subset 4 C^\circ \, . $$
		Then, \eqref{ccirc} implies that $\{u_\alpha^*\}_\alpha \subset  \frac 4 \rho \bar {\mbox{\bf B}}_{X^*}$. Hence, $\{u_\alpha^*\}_\alpha$ is bounded and thus $\{v_\alpha^*\}_\alpha$ is bounded as well.
		
		Therefore, we obtain that $u_{\alpha} \xrightarrow{w^*} u_0^*$ and $v_{\alpha} \xrightarrow{w^*} v_0^*$ (by passing to  subnets) and $x^* = u_0^* + v_0^*$. But ${C_A}^\circ$ and ${C_A}^\circ$ are $w^*$-closed and then $u_0^* \in {C_A}^\circ$, $v_0^* \in {C_B}^\circ$ and $x^* \in {C_A}^\circ + {C_B}^\circ$.
		
		We have shown that  $({C_A}^\circ+ {C_B}^\circ) \cap n \bar {\mbox{\bf B}}_{X^*}$ is $w^*$-closed for every $n\in \mathbb{N}$. By using the Banach-Dieudonné theorem (c.f. Theorem 4.4 in \cite{FA}), we obtain that ${C_A}^\circ+ {C_B}^\circ$ is $w^*$-closed.
	\end{proof}
	
	\begin{thm}[Normal intersection property with respect to Clarke normal cones]\label{ttClarke}
		Let $A$ and $B$ be closed subsets of the Banach space $X$ and let
		$x_0\in A\cap B$. Let $A$ and $B$ be tangentially transversal and let the corresponding Clarke tangent cones $\hat T_A(x_0)$ and $\hat T_B(x_0)$ be transversal.
		Then, the following holds true
		$$N_{A\cap B}(x_0) \subset N_A(x_0) + N_B(x_0) \, ,$$
		where $N_S(x)$ is the Clarke normal cone to the set $S$ at the point $x$. 
	\end{thm}
	
	\begin{proof}
		From Proposition 3.3 in \cite{BKRtt}, we have
		$$\hat T_A(x_0) \cap \hat T_B(x_0) \subset \hat T_{A\cap B}(x_0) \, . $$
		By taking the polar cones and using Proposition \ref{polars} $(ii)$ and Proposition \ref{weak_star}, we obtain
		$$N_{A\cap B}(x_0) \subset \overline{N_A(x_0) + N_B(x_0)}^{\,w^*} = N_A(x_0) + N_B(x_0) \, ,$$
		which concludes the proof of the theorem.
	\end{proof}

	\begin{corol}
		Let $A$ and $B$ be closed subsets of the Banach space $X$ and let
		$x_0\in A\cap B$. 
		Let $A$ and $B$ be strongly tangentially transversal at $x_0$.
		Then, the following hold true
		\begin{enumerate}[(i)]
			\item $\hat T_A(x_0) \cap \hat T_B(x_0) \subset \hat T_{A\cap B}(x_0) $
			\item $N_{A\cap B}(x_0) \subset N_A(x_0) + N_B(x_0) \, .$
		\end{enumerate}
	\end{corol}
	
	In order to apply the obtained normal intersection properties in some particular cases, we will need some technical results. 
	
	\begin{definition}[Definition 4.1.5 from \cite{LY}]\label{fincodim}
		A subset $S$ of the Banach space $X$ is said to be finite codimensional in $X$, if there exists a point  $x_0\in \overline{co}\, S$ such that $\overline{span}\,\{S-x_0\}$ is a finite codimensional subspace of $X$ and $\overline{co}\,\{S-x_0\}$ has a nonempty interior in this subspace.
	\end{definition}
	
	We are going to use a concept introduced in \cite{KRTs} which
	extents the notion of  a finite codimensional subset of $X$:
	
	\begin{definition}\label{solid}
		Let $X$ be a Banach space and $S$ be a  subset of $X$. The set $S$
		is said to be  quasisolid if its closed convex hull
		$\overline{\mbox{\rm co}} \ S $ has nonempty interior in its
		closed affine hull, i.e. if  there exists a point  $x_0  \in
		\overline{\mbox{\rm co}}$ ${S}$ such that $\overline{\mbox{\rm
				co}} \ \{S -x_0\} $ has nonempty interior in $\overline{\mbox{\rm
				span}}$ $(S - x_0)$ (the closed subspace spanned by $ S-x_0 $).
	\end{definition}

	\begin{prop}\label{quasis}
		Let $A$ and $B$ be closed subsets of the Banach space $X$ and let
		$x_0\in A\cap B$.  Let $D_A(x_0)$ and $D_B(x_0)$ be bounded subsets of $X$ that generate the corresponding Clarke tangent cones $\hat T_A(x_0)$ and $\hat T_B(x_0)$ to $A$ and $B$ at $x_0$. Let the set $D_A(x_0) - D_B(x_0)$ be quasisolid. Let $\hat T_A(x_0)-\hat T_B(x_0)$ be dense in $X$. 
		Then, $\hat T_A(x_0)$ and $\hat T_B(x_0)$ are transversal.
		
		Moreover, if $D_A(x_0)$ is  an  uniform tangent set to $A$ at the point $x_0$ and $D_B(x_0)$ is an  uniform  tangent set    to the set $B$ at the point $x_0$, then $A$ and $B$ are strongly tangentially transversal.
	\end{prop}
	
	\begin{proof}
		We can assume without loss of generality that the sets $D_A(x_0)$ and $D_B(x_0)$ are closed, convex, contain the origin and are contained in the unit ball. We will verify that the set $D := \overline{D_A(x_0) - D_B(x_0)}$ contains a neighbourhood of zero.
		
		We consider two cases:\\
		(i) The origin belongs to the interior of $D$\\
		(ii) The origin does not belong to the interior of $D$ or $\mbox{int} \, D = \emptyset$.

		In case (ii), either $\mathbf{0} \notin \mbox{int} \, D \neq \emptyset$ and we can separate them; or if $\mbox{int} \, D = \emptyset$ due to the quasisolidity of $D$ we have that $\overline{span}\, D = \overline{aff}\, D \neq X$. In both cases, there exists a nonzero $\xi \in X^*$ such that
		\begin{equation*}
		\langle \xi, v_1 - v_2  \rangle \le\ 0
		\end{equation*}
		for each $v_1\in D_A(x_0)$ and for each $v_2 \in D_B(x_0)$.
		The assumption that $\mathbf{0} \in D_A(x_0) \cap D_B(x_0)$ implies
		$$\langle \xi, v_1  \rangle \le\ 0 \le \langle \xi, v_2  \rangle$$
		and therefore
		\begin{equation*}
		\langle \xi, u_1 - u_2  \rangle \le\ 0
		\end{equation*}
		for each $u_1\in \hat T_A(x_0)$ and for each $u_2 \in \hat T_B(x_0)$. We have obtained that $\hat T_A(x_0)-\hat T_B(x_0)$ is contained in the closed halfspace $\{x\in X \ | \ \langle \xi, x  \rangle \le\ 0 \}$, which is a contradiction with $\overline{\hat T_A(x_0)-\hat T_B(x_0)} = X$.
		
		Therefore, case (i) holds true and then
		$$\rho \bar {\mbox{\bf B}} \subset \overline{D_A(x_0) - D_B(x_0)} \subset \overline{(\hat T_A(x_0) \cap \bar{\mbox{\bf B}}) - (\hat T_B(x_0) \cap \bar{\mbox{\bf B}}) } \, $$
		for some $\rho>0$.
	\end{proof}

	Till the end of the section, we will apply the obtained results to some known classes of sets. To this end, we will need the definitions below.
	
	\begin{definition}[Borwein and Strojwas, \cite{BS85}]
		Let $A$ be a closed subset of the Banach space $X$ and $x_0\in A$. We say that $A$ is compactly epi-Lipschitz (massive) at $x_0$, if there exist $\varepsilon>0$, $\delta >0$ and a compact set $K \subset X$, such that for all $x \in A \cap (x_0 + \delta \bar {\mbox{\bf B}})$, for all $v \in X$, $\|v\|\le \varepsilon$ and for all $t \in [0, \delta]$, there exists $k \in K$, for which $x + t(v-k)\in A$.
	\end{definition}
	
	The following definition is equivalent to the original definition of epi-Lipschitz-like sets in \cite{Borw87}:
	
	\begin{definition}
		Let $S$ be a closed subset of the Banach space $X$ and $x \in S$. The set $S$ is said to be epi-Lipschitz-like at $x$ if there exists $\delta >0$, a convex set $\Omega$ which is finite codimensional in $X$, and $\eta >0$ such that 
		$$y + t\Omega \subset S \mbox{   for all  }
		y \in (x+\delta \bar{\mbox{\bf B}})\cap S, \ t \in (0, \eta) \, .$$
	\end{definition}
	
	\begin{prop}\label{finite}
		Let $X$ be a Banach space and $S \subset \tilde S$ be subsets of $X$. If $S$ is finite codimensional, then $\tilde S$ is finite codimensional as well.
	\end{prop}
	
	\begin{proof}
		The idea of the proof is the same as in the proof of Proposition 4.3.4 from \cite{LY}. We provide it here for completeness.
		
		By Proposition 4.3.1 from \cite{LY} we have that the set $S$ is finite codimensional if and only if there exists $x_0 \in \overline{co}\, S$ such that the space\\ $X_1:=\overline{span}\,(S-x_0)$ is a finite codimensional subspace of $X$ and
		$$\bar {\mbox{\bf B}}_\delta (\mathbf{0}) \cap X_1 \subset \overline{co}\,(S-x_0) \, . $$	
		
		Let us denote $X_0:= \overline{span}\,(\tilde S - x_0)$. Since $X_0 \supset X_1$, we have that $X_0$ is a  finite codimensional subspace of $X$. If $X_0=X_1$, we are done. Otherwise, there exist linearly independent $x_1, x_2, \dots, x_k \in (\tilde S-x_0) \setminus X_1$, such that
		\begin{equation}\label{span}
		X_0:= \overline{span} \{ X_1, x_1, x_2, \dots, x_k  \} \, .
		\end{equation}
		
		We have that $\sum_{i=1}^{k} \lambda_i x_i \in \overline{co}\,(\tilde S - x_0)$ for all $\lambda_i\ge 0$ with $\sum_{i=1}^{k} \lambda_i \le 1$. Thus, for all $y\in\bar {\mbox{\bf B}}_\delta (\mathbf{0})\cap X_1$, for all $\lambda_i\ge 0$ with $\sum_{i=1}^{k} \lambda_i \le 1$ and for all $\lambda, \mu \ge 0$ with $\lambda+\mu = 1$, we obtain that
		$$\lambda y+ \mu \sum_{i=1}^{k} \lambda_i x_i \in \lambda\, \overline{co}\,(S-x_0)+ \mu\, \overline{co}\,(\tilde S - x_0) \subset \overline{co}\,(\tilde S - x_0) \, . $$
		Then, using \eqref{span}, it follows that for some $\bar \delta >0$
		$$\bar {\mbox{\bf B}}_{\bar\delta} (\bar x) \cap X_0 \subset \overline{co}\,(\tilde S-x_0) \, , $$	
		where $\bar x:= \frac{1}{k}\sum_{i=1}^{k} x_i \in \overline{co}\,(\tilde S - x_0)$. This completes the proof.
	\end{proof}
	
	\begin{corol}\label{ell}
		Let $A$ and $B$ be closed subsets of the Banach space $X$ and let
		$x_0\in A\cap B$. Let $A$ be epi-Lipschitz-like at $x_0$. Let $\hat T_A(x_0)-\hat T_B(x_0)$ be dense in $X$. Then, $A$ and $B$ are  tangentially transversal and have normal intersection property.
	\end{corol}
	
	\begin{proof}
		The tangential transversality follows from Theorem 4.3 in \cite{BKRtt}, since epi-Lipschitz-like sets are compactly epi-Lipschitz (Proposition 3.1 in \cite{Borw87}) and therefore almost massive.
		
		For the normal intersection property, first we note that the set $\Omega$ from Definition \ref{ell} is obviously contained in $\hat T_A(x_0)$. Therefore, $\hat T_A(x_0)$ is finite codimensional due to Proposition \ref{finite}. As it is a closed convex cone,\\ $\hat T_A(x_0)\cap  \bar{\mbox{\bf B}}$ is finite codimensional as well. Applying Proposition \ref{finite} once again, we obtain finite codimensionality of $D:= \overline{\hat T_A(x_0)\cap \bar{\mbox{\bf B}} - \hat T_B(x_0)\cap  \bar{\mbox{\bf B}}} $. An application of Proposition \ref{quasis} and Theorem \ref{ttClarke} completes the proof.
	\end{proof}

	\begin{example}\label{CEL}
		
		We are going to construct an example for closed subsets $A$ and $B$ of the Banach space $X:=l_\infty(\mathbb{N})$, $A$ -- compactly epi-Lipschitz at $x_0 \in A \cap B$, such that $\overline{\hat T_A(x_0) -\hat T_B(x_0)} = X$ but $ N_A(x_0) + N_B(x_0)$ is not $w^*$-closed (this $\hat T_A(x_0)$ and $\hat T_B(x_0)$ are not transversal).
		
		The set $A$ will be the same as in Example 4.1 in \cite{Borw87}:
		$$A := \{x\in X \ | \ f(x) \le 0 \} \, $$
		where $f(x) := \liminf_{n\to \infty} |x_n|$. In \cite{Borw87} it is shown that $A$ is compactly epi-Lipschitz and $\hat T_A(\mathbf{0})=c_0(\mathbb{N})$.
		
		The set $B$ will be from the following lemma:
		\begin{lem}\label{conB}
			There exists a closed subspace $B$ of $l_\infty(\mathbb{N})$ such that $c_0(\mathbb{N})+B$ is dense in $l_\infty(\mathbb{N})$ but $c_0(\mathbb{N})+B \neq l_\infty(\mathbb{N})$.
		\end{lem}
		
		\begin{proof}
			Due to Theorem 5.83 on p. 200 in \cite{Haj08} there exists a closed subspace $B$ of $l_\infty(\mathbb{N})$ such that $c_0(\mathbb{N})+B$ is dense in $l_\infty(\mathbb{N})$ and $c_0(\mathbb{N}) \cap B =  \{\mathbf{0}\}$.
			
			Let us assume that $c_0(\mathbb{N})+B = l_\infty(\mathbb{N})$. Since  $c_0(\mathbb{N}) \cap B =  \{\mathbf{0}\}$, we have that this is a direct sum: 
			\begin{equation}\label{compl}
			c_0(\mathbb{N}) \oplus B = l_\infty(\mathbb{N}) \, .
			\end{equation}
			Due to Proposition 5.3 in \cite{FA}, \eqref{compl} implies that $c_0(\mathbb{N})$ is complemented in $l_\infty(\mathbb{N})$, which is a contradiction with Theorem 5.15 in \cite{FA}.
			
		\end{proof}
		
		We have that
		$$
		\hat T_A(\mathbf{0}) -\hat T_B(\mathbf{0}) = c_0(\mathbb{N}) + B \subsetneq l_\infty(\mathbb{N})
		$$
		and $\overline{\hat T_A(\mathbf{0}) -\hat T_B(\mathbf{0})} = l_\infty(\mathbb{N})$. Therefore, $\hat T_A(\mathbf{0}) -\hat T_B(\mathbf{0})$ is not closed and due to Corollary 3 in \cite{Jam72}, this is equivalent to
		$$(\hat T_A(\mathbf{0}))^\circ + (-\hat T_B(\mathbf{0}))^\circ =  N_A(\mathbf{0}) + N_B(\mathbf{0}) $$
		being not $w^*$-closed.
	\end{example}

	\section{Sum rule for the Clarke subdifferential}
	Let $f_1: X  \rightarrow \mathbb{R}\cup \{+\infty\}$ and  $f_2: X  \rightarrow \mathbb{R}\cup \{+\infty\}$ be lower semicontinuous and proper and  $x_0 \in X$ be in $dom f_1 \cap dom f_2$. We are going to apply the results from the previous section to the closed sets
	$$C_1 := \{ (x, r_1, r_2) \in X\times \mathbb{R}\times \mathbb{R} \ | \ r_1 \ge f_1(x) \}$$
	and
	$$C_2 := \{ (x, r_1, r_2) \in X\times \mathbb{R}\times \mathbb{R} \ | \ r_2 \ge f_2(x) \} $$
	in order to obtain a sum rule for the Clarke subdifferential. This is the approach introduced by Ioffe in \cite{Ioffe84}. 
	
	\begin{thm}\label{sum_rule}
		Let the sets $C_1$ and $C_2$ have normal intersection property at $(x_0, f_1(x_0), f_2(x_0))$.
		Then
		$$\partial_C (f_1+f_2)(x_0) \subset \partial_C f_1(x_0) + \partial_C f_2(x_0) \, ,$$
		where $\partial_C$ is the Clarke subdifferential.
	\end{thm}
	
	\begin{proof}
		The normal intersection property of $C_1$ and $C_2$ at $(x_0, f_1(x_0), f_2(x_0))$ means that
		$$ N_{C_1\cap C_2}(x_0, f_1(x_0), f_2(x_0)) \subset N_{C_1}(x_0, f_1(x_0), f_2(x_0)) + N_{C_2}(x_0, f_1(x_0), f_2(x_0)) \, .$$
		It is direct that
		$$C_1 \cap C_2 \subset C:= \{ (x, r_1, r_2) \in X\times \mathbb{R}\times \mathbb{R} \ | \ r_1+r_2 \ge f_1(x) +  f_2(x) \} \, $$
		and using Lemma 2.4 from \cite{JouThi95} (with $g$ -- the identity, $C := C_1$, $D:=C_2$ and $B:=C$) we have that
		$$\hat T_{C_1 \cap C_2}(x_0, f_1(x_0), f_2(x_0)) \subset \hat T_{C}(x_0, f_1(x_0), f_2(x_0))\, . $$
		
		Hence,
		\begin{align}\label{incln}
		N_C(x_0, f_1(x_0), &f_2(x_0)) \subset N_{C_1\cap C_2}(x_0, f_1(x_0), f_2(x_0)) \\
		&\subset N_{C_1}(x_0, f_1(x_0), f_2(x_0)) + N_{C_2}(x_0, f_1(x_0), f_2(x_0)) \, . \nonumber
		\end{align}
		It can be easily verified that 
		$$\hat T_{C}(x_0, f_1(x_0), f_2(x_0)) = \{ (v, r_1, r_2) \in X\times \mathbb{R}\times \mathbb{R} \ | \ (v, r_1+r_2) \in \hat T_{epi (f_1+f_2)}(x_0, f_1(x_0)+ f_2(x_0)) \} \, .$$
		By polarity, we have that
		\begin{align*}
		N_{C}&(x_0, f_1(x_0), f_2(x_0)) = \hat T_{C}^\circ(x_0, f_1(x_0), f_2(x_0)) \\
		& =\{ (x^*, s_1, s_2) \in X^*\times \mathbb{R}\times \mathbb{R} \ | \ \langle x^* v\rangle + s_1 r_1 + s_2 r_2 \le 0\\
		& \quad \quad \quad \mbox{ for all } (v, r_1+r_2) \in \hat T_{epi (f_1+f_2)}(x_0, f_1(x_0)+ f_2(x_0)) \} \, .
		\end{align*}
		Therefore, due to the symmetry of $r_1$ and $r_2$, we obtain that
		$$N_{C}(x_0, f_1(x_0), f_2(x_0)) = \{ (x^*, s, s) \in X^*\times \mathbb{R}\times \mathbb{R} \ | \ (x^*, s) \in N_{epi (f_1+f_2)}(x_0, f_1(x_0)+ f_2(x_0)) \} \, .$$
		
		Hence,
		$$x^* \in \partial_C (f_1+f_2)(x_0) \iff (x^*, -1, -1) \in N_C(x_0, f_1(x_0), f_2(x_0)) $$
		and by using \eqref{incln} and that
		$$N_{C_1}(x_0, f_1(x_0), f_2(x_0)) = \{ (x^*, s_1, 0) \in X^* \times \mathbb{R}\times \mathbb{R} \ | \  (x^*, s_1) \in N_{epi f_1}(x_0, f_1(x_0)) \} $$
		and 
		$$N_{C_2}(x_0, f_1(x_0), f_2(x_0)) = \{ (x^*, 0, s_2) \in X^* \times \mathbb{R}\times \mathbb{R} \ | \  (x^*, s_2) \in N_{epi f_2}(x_0, f_2(x_0)) \} \, , $$
		the proof is complete.
		
	\end{proof}
	
	\begin{corol}
		Let the sets $C_1$ and $C_2$ be strongly tangentially transversal at $(x_0, f_1(x_0), f_2(x_0))$.
		Then
		$$\partial_C (f_1+f_2)(x_0) \subset \partial_C f_1(x_0) + \partial_C f_2(x_0) \, .$$
	\end{corol}
	
	\begin{corol}
		Let the sets $C_1$ and $C_2$ be tangentially transversal at $(x_0, f_1(x_0), f_2(x_0))$ and let $\hat T_{C_1}(x_0, f_1(x_0), f_2(x_0))$ and $\hat T_{C_2}(x_0, f_1(x_0), f_2(x_0)) $ be transversal.
		Then
		$$\partial_C (f_1+f_2)(x_0) \subset \partial_C f_1(x_0) + \partial_C f_2(x_0) \, .$$
	\end{corol}

	\begin{definition}[Borwein and Strojwas, \cite{BS85}]
		The function $f: X  \rightarrow [-\infty, +\infty\}$ is said to be Lipschitz-like at $x_0$, if $epi f$ is epi-Lipschitz-like at  $(x_0, f(x_0))$.
	\end{definition}
	
	\begin{corol}\label{sumLL}
		Let $f_1$ be Lipschitz-like at $x_0$ and
		\begin{equation}\label{qual}
		\{\partial_C^\infty f_1(x_0) \} \cap \{-\partial_C^\infty f_2(x_0) \} = \{ \mathbf{0} \} \, ,
		\end{equation}
		where  $\partial_C^\infty$ is the Clarke singular subdifferential. 
		Then 
		$$\partial_C (f_1+f_2)(x_0) \subset \partial_C f_1(x_0) + \partial_C f_2(x_0) \, .$$
	\end{corol}
	
	\begin{proof}
		We have that $C_1$ is epi-Lipschitz-like at $(x_0, f_1(x_0), f_2(x_0))$, since $f_1$ is Lipschitz-like at $x_0$. The qualification condition \eqref{qual} is equivalent to 
		\begin{equation}\label{tang}
		\overline{\hat T_{C_1}(x_0, f_1(x_0), f_2(x_0))- \hat T_{C_2}(x_0, f_1(x_0), f_2(x_0))} = X\times \mathbb{R}\times \mathbb{R} \, .
		\end{equation}
		due to Lemma 4.5 in \cite{BKRtt}.
		
		We can use Corollary \ref{ell} to obtain normal intersection property of $C_1$ and $C_2$ at $(x_0, f_1(x_0), f_2(x_0))$. The corollary follows from Theorem \ref{sum_rule}. 
	\end{proof}

\end{document}